\documentclass[12pt,a4paper]{article}
\usepackage{amsmath,amstext,amssymb,amscd, color}
\usepackage{graphicx}

\newcommand{\Xcomment}[1]{}

\newtheorem{theorem}{Theorem}[section]
\newtheorem{lemma}[theorem]{Lemma}
\newtheorem{corollary}[theorem]{Corollary}

\newtheorem{prop}[theorem]{Proposition}

\newcommand{\SEC}[1]{\ref{sec:#1}}  

\makeatletter \@addtoreset{equation}{section} \makeatother

\newenvironment{proof}{\noindent{\underline{\sl Proof.}}~}%
{\hfill$\qed$\medskip}

\def\qed{ \ \vrule width.1cm height.3cm depth0cm}

\newenvironment{numitem1}{\refstepcounter{equation}\begin{enumerate}%
\item[(\thesection.\arabic{equation})]}{\end{enumerate}}

\newcommand{\refeq}[1]{(\ref{eq:#1})}  

 \makeatletter
\renewcommand{\section}{\@startsection{section}{1}{0pt}%
{-3.5ex plus -1ex minus -.2ex}{2.3ex plus .2ex}%
{\normalfont\Large}}
 \makeatother

 \makeatletter
\renewcommand{\subsection}{\@startsection{subsection}{2}{0pt}%
{-3.0ex plus -1ex minus -.2ex}{1.5ex plus .2ex}%
{\normalfont\normalsize\bf}}
 \makeatother

 \makeatletter
\renewcommand{\subsubsection}{\@startsection{subsubsection}{2}{0pt}%
{-2.0ex plus -1ex minus -.2ex}{-2.0ex plus .2ex}%
{\normalfont\normalsize\underline}}
 \makeatother

\def\Rset{{\mathbb R}}
\def\Zset{{\mathbb Z}}

\def\Ascr{{\cal A}}

\def\Dscr{{\cal D}}
\def\Escr{{\cal E}}
\def\Fscr{{\cal F}}

\def\Hscr{{\cal H}}

\def\Lscr{{\cal L}}
\def\Mscr{{\cal M}}

\def\Pscr{{\cal P}}

\def\Rscr{{\cal R}}

\def\Xscr{{\cal X}}

\def\tilde{\widetilde}
\def\hat{\widehat}
\def\bar{\overline}
\def\eps{\varepsilon}

\def\dist{{\rm dist}}
\def\inter{{\rm int}}

\begin{document}

\title{A combinatorial algorithm for the planar multiflow problem with demands
located on three holes}

\author
{
    Maxim A. Babenko
    \thanks {
        Higher School of Economics;
        20, Myasnitskaya, 101000 Moscow, Russia;
        email: \texttt{maxim.babenko@gmail.com}.
    }
    \and
    Alexander V. Karzanov
    \thanks {
        Institute for System Analysis of the RAS;
        9, Prospect 60 Let Oktyabrya, 117312 Moscow, Russia;
        email: \texttt{sasha@cs.isa.ru}.
    }
}

\date{}

\maketitle

 \begin{abstract}
We consider an undirected multi(commodity)flow demand problem in which a supply
graph is planar, each source-sink pair is located on one of three specified
faces of the graph, and the capacities and demands are integer-valued and
Eulerian. It is known that such a problem has a solution if the cut and
(2,3)-metric conditions hold, and that the solvability implies the existence of
an integer solution. We develop a purely combinatorial strongly polynomial
solution algorithm.
 \end{abstract}

 \smallskip



\section{Introduction} \label{sec:intr}

Among a variety of multi(commodity)flow problems, one popular class embraces
multiflow demand problems in undirected planar graphs in which the demand pairs
are located within specified faces of the graph. More precisely, a problem
input consists of: a planar graph $G=(V,E)$ with a fixed embedding in the
plane; nonnegative integer \emph{capacities} $c(e)\in\Zset_+$ of edges $e\in
E$; a subset $\Hscr\subseteq \Fscr_G$ of faces, called \emph{holes} (where
$\Fscr_G$ is the set of faces of $G$); a set $D$ of pairs $st$ of vertices such
that both $s,t$ are located on (the boundary of) one of the holes; and
\emph{demands} $d(st)\in\Zset_+$ for $st\in D$. A \emph{multiflow} for $G,D$ is
meant to be a pair $f=(\Pscr,\lambda)$ consisting of a set $\Pscr$ of $D$-paths
$P$ in $G$ and nonnegative real weights $\lambda(P)\in\Rset_+$. Here a path $P$
is called a $D$-\emph{path} if $\{s_P,t_P\}=\{s,t\}$ for some $st\in D$, where
$s_P$ and $t_P$ are the first and last vertices of $P$, respectively. We call
$f$ \emph{admissible} for $c,d$ if it satisfies the capacity constraints:
  \begin{equation} \label{eq:capac}
  \sum\Bigl(\lambda(P)\colon e\in P\in \Pscr\Bigr)\le c(e),\qquad e\in E,
   \end{equation}
and realizes the demands:
  \begin{equation} \label{eq:demand}
  \sum\Bigl(\lambda(P)\colon P\in\Pscr,\> \{s_P,t_P\}=\{s,t\}\Bigr)=d(st),
\qquad st\in D.
  \end{equation}

The (fractional) \emph{demand problem}, denoted as $\Dscr(G,\Hscr,D,c,d)$, or
$\Dscr(c,d)$ for short, is to find an admissible multiflow for $c,d$ (or to
declare that there is none). When the number of holes is ``small'', this linear
program is known to possess nice properties. To recall them, we need some
terminology and notation.

For $X\subseteq V$, the set of edges of $G$ with one end in $X$ and the other
in $V-X$ is denoted by $\delta(X)=\delta_G(X)$ and called the \emph{cut} in $G$
determined by $X$. We also denote by $\rho(X)=\rho_D(X)$ the set of pairs
$st\in D$ \emph{separated} by $X$, i.e., such that $|\{s,t\}\cap X|=1$. For a
singleton $v$, we write $\delta(v)$ for $\delta(\{v\})$, and $\rho(v)$ for
$\rho(\{v\})$. For a function $g:S\to\Rset$ and a subset $S'\subseteq S$,
~$g(S')$ denotes $\sum(g(e)\colon e\in S')$. So $c(\delta(X))$ is the capacity
of the cut $\delta(X)$, and $d(\rho(X))$ is the total demand on the elements of
$D$ separated by $X$.

A capacity-demand pair $(c,d)$ is said to be \emph{Eulerian} if
$c(\delta(v))-d(\rho(v))$ is even for all vertices $v\in V$.

The simplest sort of necessary conditions for the solvability of problem
$\Dscr(c,d)$ (with any $G,D$) is the well-known \emph{cut condition}:
    \begin{equation} \label{eq:cut_cond}
  \Delta_{c,d}(X):=c(\delta(X))-d(\rho(X))\ge 0
   \end{equation}
should hold for all $X\subset V$. It need not be sufficient, and in general the
solvability of a multiflow demand problem is provided by metric conditions. In
our case the following results have been obtained. \medskip

(A) For $|\Hscr|=1$, Okamura and Seymour~\cite{OS} showed that the cut
condition is sufficient, and that if $(c,d)$ is Eulerian and the problem
$\Dscr(c,d)$ has a solution, then it has an integer solution, i.e., there
exists an admissible multiflow $(\Pscr,\lambda)$ with $\lambda$ integer-valued.
Okamura~\cite{ok} showed that these properties continue to hold if $|\Hscr|=2$.
\medskip

(B) For $|\Hscr|=3$, the cut condition becomes not sufficient and the
solvability criterion involves also the so-called \emph{(2,3)-metric
condition}. It is related to a map $\sigma:V\to V(K_{2,3})$, where $K_{p,q}$ is
the complete bipartite graph with parts of $p$ and $q$ vertices. Such a
$\sigma$ defines the \emph{metric} $m=m^\sigma$ on $V$ by
$m(u,v):=\dist(\sigma(u),\sigma(v))$, $u,v\in V$, where $\dist$ denotes the
distance (the shortest path length) between vertices in $K_{2,3}$. It gives a
partition of $V$ into five sets (with distances 1 or 2 between them), and $m$
is said to be a \emph{(2,3)-metric} on $V$. We denote $\sum(c(e)m(e)\colon e\in
E)$ by $c(m)$, and $\sum(d(st)m(st)\colon st\in D)$ by $d(m)$. Karzanov showed
the following.
  \begin{theorem}[\cite{K3hII}] \label{tm:3hole}
Let $|\Hscr|=3$. Then $\Dscr(c,d)$ has a solution if and only if the cut
condition~\refeq{cut_cond} holds and
  \begin{equation} \label{eq:23metr}
  \Delta_{c,d}(m):=c(m)-d(m)\ge 0
  \end{equation}
holds for all (2,3)-metrics $m$ on $V$ (the \emph{(2,3)-metric condition}).
Furthermore, if $(c,d)$ is Eulerian and the problem $\Dscr(c,d)$ has a
solution, then it has an integer solution.
  \end{theorem}

We call $\Delta_{c,d}(X)$ in~\refeq{cut_cond} (resp. $\Delta_{c,d}(m)$
in~\refeq{23metr}) the \emph{excess} of a set $X$ (resp. a (2,3)-metric $m$)
w.r.t. $c,d$. One easily shows that $\Delta_{c,d}(X)$ and $\Delta_{c,d}(m)$ are
even if $(c,d)$ is Eulerian.
\medskip

(C) When $|\Hscr|=4$, the situation becomes more involved. As is shown
in~\cite{K4h}, the solvability criterion for $\Dscr(c,d)$ involves, besides
cuts and (2,3)-metrics, metrics $m=m^\sigma$ on $V$ induced by maps
$\sigma:V\to V(\Gamma)$ with $\Gamma$ running over a set of planar graphs with
four faces, and merely the existence of a half-integer solution is guaranteed
in a solvable Eulerian case. When $|\Hscr|=5$, the set of unavoidable metrics
in the solvability criterion becomes ugly (see~\cite[Sec.~4]{K3hI}), and the
fractionality status is unknown so far.
\medskip

In this paper we focus on algorithmic aspects. The first combinatorial strongly
polynomial algorithm (having complexity $O(n^3\log n)$) to find an integer
solution in the Eulerian case with $|\Hscr|=1$ is due to Frank~\cite{frank},
and subsequently a number of faster algorithms have been devised; a linear-time
algorithm is given in~\cite{WW}. Hereinafter $n$ stands for the number $|V|$ of
vertices of the graph. Efficient algorithms for $|\Hscr|=2$ are known as well.
For a survey and references in cases $|\Hscr|=1,2$, see, e.g.,~\cite{sch}.

Our aim is to give an algorithm to solve problem $\Dscr(c,d)$ with $|\Hscr|=3$,
which checks the solvability and finds an integer admissible multiflow in the
Eulerian case. Our algorithm uses merely combinatorial means and is strongly
polynomial (though having a high polynomial degree). It is based on a
subroutine for a certain planar version of the (2,3)-metric minimization
problem. We explain how to solve the latter efficiently and in a combinatorial
fashion, by reducing it to a series of shortest paths problems in a dual planar
graph.
\medskip

\noindent\textbf{Remark 1.} The (2,3)-metric minimization problem in a general
edge-weighted graph with a specified set of five terminals can be solved in
strongly polynomial time (by use of the ellipsoid method)~\cite{K87} or by a
combinatorial weakly polynomial algorithm~\cite{K98}. \medskip

This paper is organized as follows. Section~\SEC{prelim} reviews facts
from~\cite{K3hII} refining the structure of cuts and (2,3)-metrics that are
essential for the solvability of our 3-hole demand problem. Using these
refinements, Sections~\SEC{veri_cut} and~\SEC{veri_23} develop efficient
combinatorial procedures to verify cut and (2,3)-metric conditions for problem
$\Dscr(c,d)$ with initial or current $c,d$; moreover, these procedures
determine or duly estimate the minimum excesses of regular cuts and
(2,3)-metrics, which is important for the efficiency of our algorithm for
$\Dscr(c,d)$. This algorithm is described in Section~\SEC{alg}.
\medskip

To slightly simplify the further description, we will assume, w.l.o.g., that
the boundary of any hole $H$ contains no isthmus. For if $b(H)$ has an isthmus
$e$, we can examine the cut $\{e\}$. If it violates the cut condition, the
problem $\Dscr(c,d)$ has no solution. Otherwise $\Dscr(c,d)$ is reduced to two
smaller demand problems, with at most 3 holes and with Eulerian data each, by
deleting $e$ and modifying demands concerning $H$.


\section{Preliminaries} \label{sec:prelim}

Throughout the rest of the paper, we deal with $G=(V,E),\Hscr,D,c,d$ as above
such that $|\Hscr|=3$ and $(c,d)$ is Eulerian. Let $\Hscr=\{H_1,H_2,H_3\}$.

One may assume that the graph $G=(V,E)$ is connected and its outer (unbounded)
face is a hole (say, $H_3$). We identify objects in $G$, such as edges, paths,
subgraphs, and etc., with their images in the plane. A face $F\in\Fscr_G$ is
regarded as an open region in the plane. Since $G$ is connected, the boundary
$b(F)$ of $F$ is connected, and we identify it with the corresponding cycle
(closed path) considered up to reversing and shifting cyclically. Note that
this cycle may contain repeated vertices or edges (an edge of $G$ may be passed
by $b(F)$ twice, in different directions). A subpath in this cycle is called a
\emph{segment} in $b(F)$.

We denote the subgraph of $G$ induced by a set $X\subseteq V$ by $[X]=[X]_G$,
the set of faces of $G$ whose boundary is entirely contained in $[X]$ by
$\Fscr(X)$, and the region in the plane that is the union of $[X]$ and all
faces in $\Fscr(X)$ by $\Rscr(X)$. We also need additional terminology and
notation.
\smallskip

A subset $X\subset V$ (as well as the cut $\delta(X)$) is called \emph{regular}
if the region $\Rscr(X)$ is simply connected (i.e., it is connected and any
closed curve in it can be continuously deformed into a point), and for each
$i=1,2,3$, ~$[X]\cap \/b(H_i)$ forms a segment of $b(H_i)$. In particular, the
subgraph $[X]$ is connected.
\smallskip

Let $\{t_1,t_2\}$ and $\{s_1,s_2,s_3\}$ be the parts (color classes) in
$K_{2,3}$. Given $\sigma:V\to V(K_{2,3})$, we denote the set $\sigma^{-1}(t_i)$
by $T_i=T_i^\sigma$, and $\sigma^{-1}(s_j)$ by $S_j=S_j^\sigma$. Then
$\Xi^\sigma=(T_1,T_2,S_1,S_2,S_3)$ is a partition of $V$. The (2,3)-metric
$m^\sigma$ is called \emph{regular} if:
  \begin{numitem1}
  \begin{itemize}
\item[(i)] all sets $T_1,T_2,S_1,S_2,S_3$ in $\Xi^\sigma$  are nonempty;
\item[(ii)] for $i=1,2,3$, the region $\Rscr(S_i)$ is
simply connected;
\item[(iii)] for $i,j\in\{1,2,3\}$,  ~$S_i\cap \/b(H_j)=\emptyset$ holds if and only if
$i=j$; and for $i\ne j$, ~$[S_i]\cap \/b(H_j)$ forms a segment of $b(H_j)$.
 \end{itemize}
 \label{eq:reg23}
  \end{numitem1}
Then the complement to $\Rset^2$ of $H_1\cup H_2\cup H_3\cup
\Rscr(S_1)\cup\Rscr(S_2)\cup\Rscr(S_3)$ consists of two connected components,
one containing $T_1$ and the other containing $T_2$. The structure described
in~\refeq{reg23} is illustrated in the picture.

\vspace{-0.3cm}
\begin{center}
\includegraphics{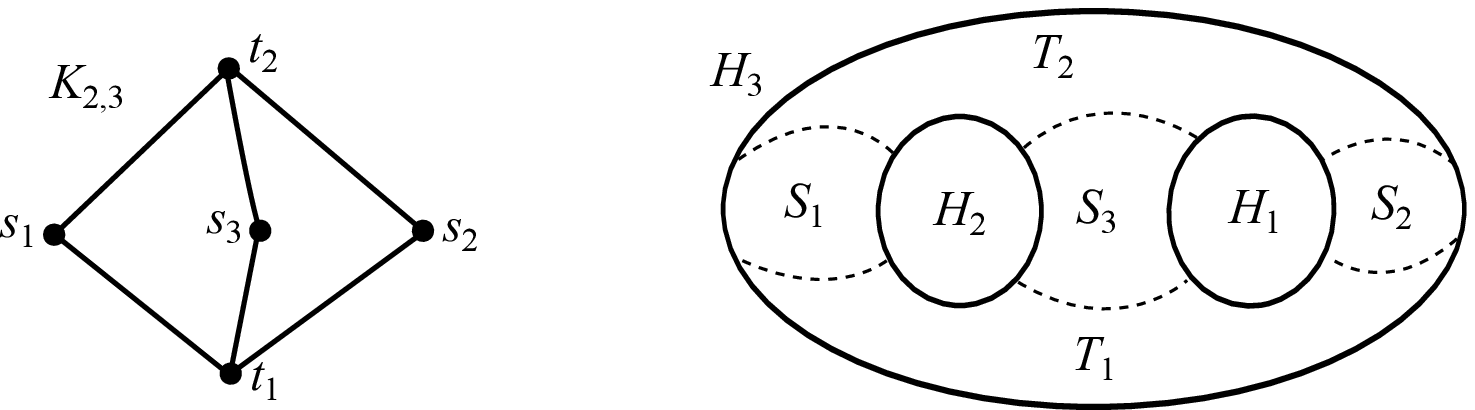}
\end{center}
\vspace{-0.3cm}

The notions of regular sets (cuts) and (2,3)-metric are justified by the
following important strengthening of the first assertion in
Theorem~\ref{tm:3hole} (cf.~\cite{K3hII}).
  \begin{theorem} \label{tm:3h_sharp}
$\Dscr(c,d)$ has a solution if and only if cut condition~\refeq{cut_cond} holds
for all regular subsets $X\subset V$ and (2,3)-metric condition~\refeq{23metr}
holds for all regular (2,3)-metrics on $V$.
  \end{theorem}

\noindent\textbf{Remark 2.} In fact, the refined solvability criterion for
$\Dscr(c,d)$ given in~\cite[Stat.~2.1]{K3hII} involves a slightly smaller set
of (2,3)-metrics (called \emph{proper} there) than that defined
by~\refeq{reg23}; also it does not specify a collection of cuts. Note, however,
that if $X\subset V$ is not regular, then one can easily find nonempty sets
$X',X''\subset V$ such that $\delta(X')\cap\delta(X'')=\emptyset$,
~$\delta(X')\cup\delta(X'')\subseteq \delta(X)$, and $\rho(X)\subseteq
\rho(X')\cup \rho(X'')$. This implies that $X$ is redundant (it can be excluded
from verification of~\refeq{cut_cond}).
\medskip


\section{Verifying the cut condition} \label{sec:veri_cut}

In this and next sections we describe efficient procedures for checking the
solvability of $\Dscr(G,\Hscr,D,c,d)$ (concerning the initial or current data).
By Theorem~\ref{tm:3h_sharp}, it suffices to verify validity of cut
condition~\refeq{cut_cond} for regular sets and (2,3)-metric
condition~\refeq{23metr} for regular (2,3)-metrics. We reduce both problems to
ones on shortest paths in a certain dual graph. Moreover, on this way we shall
obtain certain lower bounds on the minimum excesses of regular sets and regular
(2,3)-metrics, which are crucial for our algorithm.

The dual graph needed to us is constructed as follows. First we take the
standard planar dual graph $G^\ast=(V^\ast, E^\ast)$ of $G$, i.e., $V^\ast$ is
bijective to $\Fscr_G$ and $E^\ast$ is bijective to $E$, defined by
$F\in\Fscr_G\mapsto v_F\in V^\ast$ and $e\in E\mapsto e^\ast\in E^\ast$, where
a dual edge $e^\ast$ connects vertices $v_F$ and $v_{F'}$ if $F,F'$ are the
faces whose boundaries contain $e$ (possibly $F=F'$). (Usually one assumes that
$v_F$ is a point in $F$ and that $e^\ast$ crosses $e$.) We also denote the
vertex of $G^\ast$ corresponding to a hole $H_i$ by $z_i$.

Then we slightly modify $G^\ast$ as follows. For $i=1,2,3$, let $E_i$ denote
the sequence of edges of the cycle $b(H_i)$. (Recall that $b(H_i)$ has no
isthmus, as mentioned in the Introduction; hence all edges in $E_i$ are
different.) Then the dual vertex $z_i$ has degree $|E_i|$ and is incident with
the dual edges $e^\ast$ for $e\in E_i$. We split $z_i$ into $|E_i|$ vertices
$z_{i,e}$ of degree 1 each, where for $e\in E_i$, the end $z_i$ of $e^\ast$ is
replaced by $z_{i,e}$. These pendant vertices are called \emph{terminals}, they
belong to the boundary of the same face, denoted as $\hat H_i$, and the set of
these terminals ordered clockwise around $\hat H_i$ is denoted by $Z_i$.

The resulting graph is just the desired dual graph for $(G,\Hscr)$, denoted as
$\hat G^\ast$. An example of transforming $G$ into $\hat G^\ast$ in a
neighborhood of a hole $H_i$ is illustrated in the picture, where $A,\ldots,F$
are faces in $G$, and the terminals in $b(\hat H_i)$ are indicated by big
circles.

\vspace{-0.3cm}
\begin{center}
\includegraphics{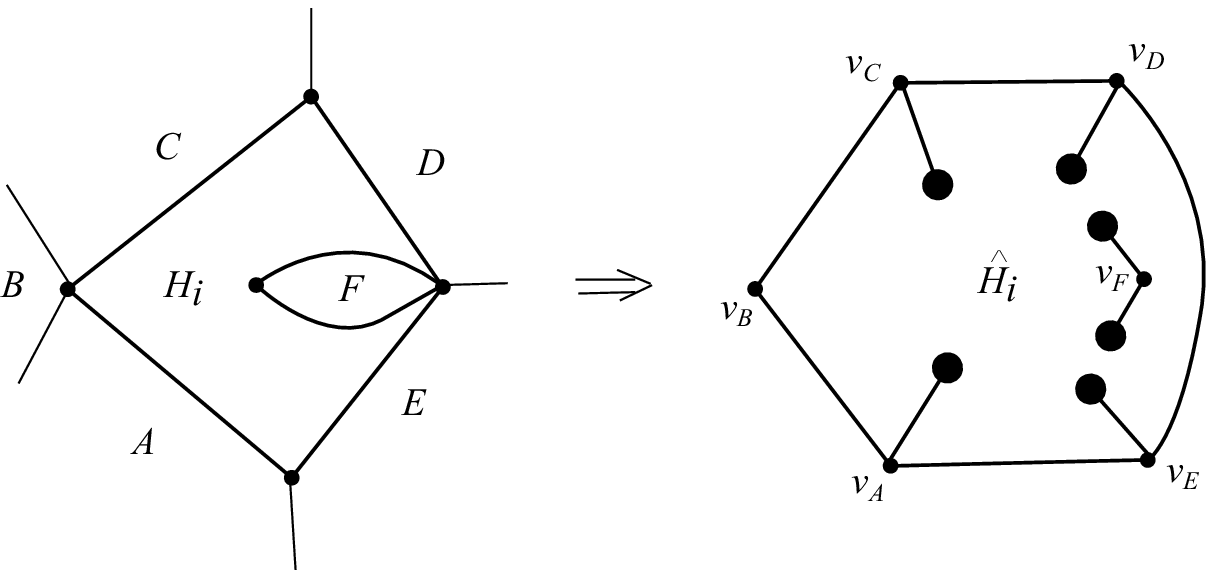}
\end{center}
\vspace{-0.3cm}

The edges of $\hat G^\ast$ have \emph{lengths} inherited from the capacities in
$G$, namely, we define $c(e^\ast):=c(e)$ for $e\in E$.
\smallskip

The rest of this section is devoted to verifying the cut condition and
estimating the minimum excesses of regular sets. \medskip

\noindent\textbf{Remark 3.} Alternatively, one can deal with subsets $X\subset
V$ subject to the only condition that for $i=1,2,3$, ~$[X]\cap b(H_i)$ is a
segment of $b(H_i)$; let us call such an $X$ \emph{semi-regular}. The minimum
excess among such sets can be computed by enumerating the triples of segments
in $b(H_1),b(H_2),b(H_3)$ and finding the corresponding minimum cut capacity
for each triple; this takes $O(n^6)$ minimum cut computations in $G$. We,
however, prefer to deal with regular sets and apply a shortest dual paths
method, which can be regarded as an introduction to the method of estimating
the minimum excess of (2,3)-metrics described in the next section.
  \medskip

Consider a regular set $X\subset V$. The fact that the region $\Rscr(X)$ is
simply connected implies that the cut $\delta(X)$ of $G$ corresponds to a
simple cycle of $G^\ast$, denoted as $C(X)$, and to a set of paths in $\hat
G^\ast$. More precisely, we say that $X$ has \emph{type} $k=|\Hscr(X)|$, where
$\Hscr(X)$ denotes the set of holes $H_i$ such that $[X]\cap
b(H_i)\ne\emptyset, b(H_i)$. Since $k=0$ implies $d(\rho(X))=0$, only sets $X$
of types 1,2,3 are essential in~\refeq{cut_cond}. For $H_i\in\Hscr(X)$,  the
cut $\delta(X)$ meets $b(H_i)$ by a pair $\{e,g\}$ of edges, denoted as
$\Pi_i(X)$ (taking into account that $b(H_i)$ has no isthmus). Let $D_i(e,g)$
denote the set of demand pairs $st\in D$ located on $b(H_i)$ and separated by
$X$ (i.e., $s,t$ lie in different components of $b(H_i)-\{e,g\}$).

Suppose that $X$ is of type 1. Let $\Hscr(X)=\{H_i\}$ and $\Pi_i(X)=\{e,g\}$.
The cycle $C(X)$ in $G^\ast$ passes the elements $e^\ast,z_i,g^\ast$. It turns
into path $P(X)$ connecting the terminals $z_{i,e}$ and $z_{i,g}$ in $\hat
G^\ast$, and we have
  $$
   c(\delta(X))=c(C(X))=c(P(X))
   $$
(regarding cycles and paths as edge sets).

Let $\Xscr(e,g)$ be the collection of regular sets $X\subset V$ of type 1 such
that $\Hscr(X)=\{H_i\}$ and $\Pi_i(X)=\{e,g\}$, and suppose that we are going
to verify~\refeq{cut_cond} and, moreover, to find the minimum excess within
this collection. The right hand side value in~\refeq{cut_cond} is constant:
$d(\rho(X))=d(D_i(e,g))$; therefore, the task is reduced to finding a
$c$-shortest path $P$ from $z_{i,e}$ to $z_{i,g}$ in $\hat G^\ast$.

Thus, verification of the cut condition for the regular sets of type 1 and,
moreover, finding the minimum excess among them, is reduced to solving $O(n)$
shortest paths problems in $\hat G^\ast$ (each handling fixed $i\in\{1,2,3\}$
and $e\in E_i$ and all $g\in E_i$) and to computing $O(n^2)$ values $D_i(e,g)$.

For $\alpha=1,2,3$, let $\mu_{c,d}^\alpha$ denote the minimum excess
$\Delta_{c,d}(X)$ among the regular sets $X\subset V$ of type $\alpha$. We have
the following

\begin{prop} \label{pr:type1}
$\mu_{c,d}^1$ can be found in time $O(n^2 +n\cdot SP(n))$, where $SP(n')$ is
the complexity of a shortest paths algorithm in a planar graph with $n'$ nodes.
  \end{prop}

To verify~\refeq{cut_cond} among the regular sets of type 2, we fix distinct
$i,j\in\{1,2,3\}$ and scan pairs $\{e,g\}\subset E_i$ and $\{e',g'\}\subset
E_j$. Let $\Xscr(e,g;e',g')$ be the collection of regular sets $X\subset V$ of
type 2 with $\Pi_i(X)=\{e,g\}$ and $\Pi_j(X)=\{e',g'\}$. For these sets $X$,
the right hand side value in \refeq{cut_cond} is again a constant, namely,
$d(D_i(e,g))+d(D_j(e',g'))=:\tilde d$. So we have to minimize $c(\delta (X))$
among $X\in\Xscr(e,g;e',g')$ and compare this minimum with $\tilde d$.

Now the cycle $C(X)$ in $G^\ast$ generates two disjoint paths $P,Q$ in $\hat
G^\ast$ going from $\{z_{i,e},z_{i,g}\}$ to $\{z_{j,e'},z_{j,g'}\}$; e.g., $P$
is a $z_{i,e}-z_{j,e'}$ path and $Q$ is a $z_{i,g}-z_{j,g'}$ path. Then
$c(\delta(X))=c(P)+c(Q)$.

This prompts an approach to computing the value
   $$
   \tilde c:=\min\{c(\delta(X))\colon X\in\Xscr(e,g;e',g')\}
   $$
or duly estimating it from below. In the graph $\hat G^\ast$ we find (simple)
$c$-shortest paths from each terminal in $\{z_{i,e},z_{i,g}\}$ to each terminal
in $\{z_{j,e'},z_{j,g'}\}$. Assume for definiteness that
  $$
  \bar c:=\dist (z_{i,e},z_{j,e'})+\dist (z_{i,g},z_{j,g'})\le
   \dist (z_{i,e},z_{j,g'})+\dist (z_{i,g},z_{j,e'})
  $$
(where we write `$\dist$' for the distance w.r.t. $c$) and let $P$ and $Q$ be
$c$-shortest paths from $z_{i,e}$ to $z_{j,e'}$ and from $z_{i,g}$ to
$z_{j,g'}$, respectively. Suppose that $P$ and $Q$ are disjoint. Then they
induce a simple cycle $C$ in $G^\ast$ with $c(C)=\bar c$, and the faces of
$G^\ast$ lying inside $C$ determine a regular set $X\in\Xscr(e,g;e',g')$ in $G$
with $c(\delta(X))=\bar c$. Then $\bar c=\tilde c$.

Next suppose that $P\cap Q\ne\emptyset$. Let $\Gamma$ be the subgraph of $\hat
G^\ast$ induced by the edges contained in exactly one of $P,Q$. The vertices
$z_{i,e},z_{i,g},z_{j,e'},z_{j,g'}$ are of degree 1 and all other vertices of
$\Gamma$ have even degrees. Hence we can find in $\Gamma$ two simple paths
$P',Q'$ such that either (a) each of $P',Q'$ connects $\{z_{i,e},z_{i,g}\}$ and
$\{z_{j,e'},z_{j,g'}\}$, and $P',Q'$ are \emph{disjoint}, or (b) $P'$ connects
$z_{i,e}$ and $z_{i,g}$, ~$Q'$ connects $z_{j,e'}$ and $z_{j,g'}$, and $P',Q'$
are \emph{edge-disjoint}. In particular, $c(P')+c(Q')\le \bar c\le \tilde c$.

Case~(a) is similar to the one considered above. In case~(b), $P',Q'$ induce
simple cycles $C,C'$ in $G^\ast$, where $C$ passes $e^\ast,z_i,g^\ast$, ~$C'$
passes $e'^\ast,z_j,g'^\ast$, and $c(C)+c(C')=c(P')+c(Q')\le\tilde c$. Now the
faces of $G^\ast$ lying inside $C$ (resp. $C'$) determine a regular set
$X\in\Xscr(e,g)$ (resp. $X'\in\Xscr(e',g')$) of type 1 in $G$. Then
  $$
c(\delta(X))=c(C),\;\; c(\delta(X'))=c(C'),\;\; d(\rho(X))=d(D_i(e,g)), \;\;
d(\rho(X'))=d(D_j(e',g')).
  $$
This implies that if~\refeq{cut_cond} is violated for some set in
$\Xscr(e,g;e',g')$, i.e., $\tilde c<\tilde d$, then so is for at least one of
$X,X'$ either. Moreover, we obtain the following

\begin{prop} \label{pr:type2}
By applying the above procedure to all $e,g\in E_i$ and $e',g'\in E_j$, $i\ne
j$, one can find, in time $O(n^4 +n\cdot SP(n))$, a bound $\nu_{c,d}^2\le
\mu_{c,d}^2$ for which at least one of the following is true:

{\rm (i)} $\nu_{c,d}^2= \mu_{c,d}^2$;

{\rm(ii)} there are two regular sets $X,Y$ of type 1 such that
$\Hscr(X)\ne\Hscr(Y)$ and $\Delta_{c,d}(X)+\Delta_{c,d}(Y)= \nu_{c,d}^2$.
  \end{prop}

Finally, to verify~\refeq{cut_cond} among the regular sets of type 3 we scan
all triples of pairs $\{e^i,g^i\}\subset E_i$, $i=1,2,3$. Let
$\Xscr=\Xscr(e^1,g^1;e^2,g^2;e^3,g^3)$ be the collection of corresponding
regular sets related to such a six-tuple. As before, we have a constant in the
right hand side of~\refeq{cut_cond}, namely, $\tilde d:=\sum(d(D_i(e^i,g^i))
\colon i=1,2,3)$, and the goal is to find or duly estimate from below the
minimum cut capacity
  $$
  \tilde c:=\min\{c(\delta(X))\colon X\in \Xscr\}.
  $$

Acting as in the previous case, we reduce the task to finding in $\hat G^\ast$
$c$-shortest paths from each of $\{z_{i,e^i},z_{i,g^i}\}$ to each of
$\{z_{j,e^j},z_{j,g^j}\}$ for all $i<j$. Among these, we take three paths
$P_1,P_2,P_3$ with the minimum total $c$-length such that all endvertices of
these paths are different, and each path connects the boundaries of different
holes; let for definiteness $P_i$ connects $z_{i,e^i}$ and $z_{i+1,g^{i+1}}$
(taking indices modulo 3).

Comparing cuts $\delta(X)$, $X\in\Xscr$, with their counterparts (path systems)
in $\hat G^\ast$, we have
  $$
  \bar c:=c(P_1)+c(P_2)+c(P_3)\le \tilde c.
  $$
Moreover, using $P_1,P_2,P_3$, one can construct a ``checker''
for~\refeq{cut_cond} which is at least as strong as the whole collection
$\Xscr$.

This is immediate when $P_1,P_2,P_3$ are pairwise disjoint. And if not, we
proceed similarly to the previous case. More precisely, let $\Gamma$ be the
subgraph of $\hat G^\ast$ induced by the edges that belong to an odd number of
paths among $P_1,P_2,P_3$. Since all nonterminal vertices in $\Gamma$ have even
degrees, we can find in $\Gamma$ three simple paths $P'_1,P'_2,P'_3$ such that:
either

(a) each $P'_i$ connects $\{z_{i,e^i},z_{i,g^i}\}$ and
$\{z_{i+1,e^{i+1}},z_{i+1,g^{i+1}}\}$, and $P'_1,P'_2,P'_3$ are pairwise
disjoint, or

(b) each $P'_i$ connects $z_{i,e^i}$ and $z_{i,g^i}$, and $P'_1,P'_2,P'_3$ are
pairwise edge-disjoint, or

(c) for some $\{i,j,k\}=\{1,2,3\}$, ~$P'_i$ connects $z_{i,e^i}$ and
$z_{i,g^i}$, each of $P'_j,P'_k$ connects $\{z_{j,e^j},z_{j,g^j}\}$ and
$\{z_{k,e^k},z_{k,g^k}\}$, the paths $P'_j,P'_k$ are disjoint and they are
edge-disjoint from $P'_i$.

In particular, $c(P'_1)+c(P'_2)+c(P'_3)\le \bar c\le \tilde c$. In case~(a), we
obtain a required set $X\in\Xscr$. In case~(b), $P'_1,P'_2,P'_3$ induce simple
cycles $C_1,C_2,C_3$ in $G^\ast$, where $C_i$ passes
$(e^i)^\ast,z_i,(g^i)^\ast$, which in turn determine regular sets
$X_i\in\Xscr(e^i,g^i)$ of type 1 satisfying $\sum(c(\delta(X_i))\colon
i=1,2,3)\le\tilde c$ and $\sum(d(\rho(X_i))\colon i=1,2,3)=\tilde d$. And
case~(c) gives regular sets $X\in \Xscr(e^i,g^i)$ and $Y\in \Xscr(e^j,g^j;
e^k,g^k)$ such that $c(\delta(X))+c(\delta(Y)\le\tilde c$ and
$d(\rho(X))+d(\rho(Y))=\tilde d$.

This leads to the following

\begin{prop} \label{pr:type3}
By applying the above procedure to all sets of six edges $e^i,g^i\in E_i$,
$i=1,2,3$, one can find, in time $O(n^6 +n\cdot SP(n))$, a bound
$\nu_{c,d}^3\le \mu_{c,d}^3$ for which at least one of the following is true:

{\rm(i)} $\nu_{c,d}^3= \mu_{c,d}^3$;

{\rm(ii)} there are three regular sets $X_1,X_2,X_3$ of type 1 such that
$\Hscr(X_i)=\{H_i\}$ and
$\Delta_{c,d}(X_1)+\Delta_{c,d}(X_2)+\Delta_{c,d}(X_3)= \nu_{c,d}^3$;

{\rm (iii)} there are two regular sets $X,Y$ such that $\Hscr(X)=\{H_i\}$,
~$\Hscr(Y)=\{H_j,H_k\}$, where $\{i,j,k\}=\{1,2,3\}$, and
$\Delta_{c,d}(X)+\Delta_{c,d}(Y)= \nu_{c,d}^3$.
  \end{prop}

In particular, Propositions~\ref{pr:type1},~\ref{pr:type2},~\ref{pr:type3} give
the following
  \begin{corollary} \label{cor:ver_cut}
To verify validity or violation of cut condition~\refeq{cut_cond} for
$\Dscr(c,d)$ reduces to $O(n)$ shortest paths computations in a dual planar
graph plus $O(n^6)$ elementary operations.
  \end{corollary}


\section{Verifying (2,3)-metric conditions} \label{sec:veri_23}

In the procedure of verifying the (2,3)-metric condition for
$\Dscr(G,\Hscr,D,c,d)$, described in this section, we also use a technique of
shortest paths in the dual graph $\hat G^\ast$.

Consider a regular (2,3)-metric $m=m^\sigma$ and its corresponding partition
$(T_1,T_2,S_1,S_2,S_3)$ (see~\refeq{reg23}). For $i\in\{1,2,3\}$, consider the
cycle $b(H_i)$. By the regularity of $m$, this cycle shares two edges with the
cut $\delta(S_{i-1})$, denoted as $g(i-1),h(i-1)$, and two edges with
$\delta(S_{i+1})$, denoted as $g'(i+1),h'(i+1)$; let
$g(i-1),h(i-1),h'(i+1),g'(i+1)$ occur in this order clockwise in $b(H_i)$
(taking indices modulo 3). Note that, although the segments $[S_{i-1}]\cap
b(H_i)$ and $[S_{i+1}]\cap b(H_i)$ are disjoint, the edges $g(i-1)$ and
$g'(i+1)$ may coincide, and similarly for $h(i-1)$ and $h'(i+1)$.

So, for $p=1,2,3$, the cut~$\delta(S_p)$ meets $b(H_{p+1})$ by $\{g(p),h(p)\}$,
meets $b(H_{p-1})$ by $\{g'(p),h'(p)\}$, and does not meet $b(H_p)$. Since the
region $\Rscr(S_p)$ is simply connected, the cut $\delta(S_p)$ corresponds to a
simple cycle $C(S_p)$ in $G^\ast$; it passes the elements
$g(p)^\ast,z_{p+1},h(p)^\ast,h'(p)^\ast,z_{p-1},g'(p)^\ast$ (in the
counterclockwise order). The cycle $C(S_p)$ turns into two disjoint paths in
$\hat G^\ast$: path $P_p$ connecting the terminals $z_{p+1,g(p)}$ and
$z_{p-1,g'(p)}$, and path $Q_p$ connecting $z_{p+1,h(p)}$ and $z_{p-1,h'(p)}$.
See the picture.

\vspace{-0.3cm}
\begin{center}
\includegraphics{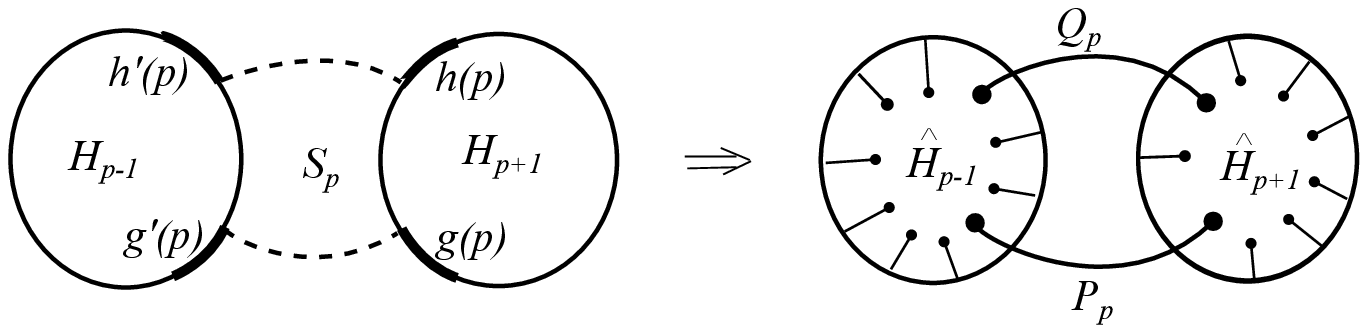}
\end{center}
\vspace{-0.3cm}

This correspondence gives $c(\delta(S_p))=c(P_p)+c(Q_p)$, implying
  $$
  c(m)=\sum\Bigl(c(\delta(S_p))\colon p=1,2,3\Bigr)
         =\sum\Bigl(c(P_p)+c(Q_p)\colon p=1,2,3\Bigr),
  $$
taking into account that no edge of $G$ connects $T_1$ and $T_2$.

In order to express $d(m)$, consider arbitrary edges $b_1,b_2,b_3,b_4$
occurring in this order in a cycle $b(H_i)$, possibly with $b_q=b_{q+1}$ for
some $q$ (letting $b_5:=b_1$). Removal of these edges from the cycle produces
four segments $\omega_1,\omega_2,\omega_3,\omega_4$, where $\omega_q$ is the
(possibly empty) segment between $b_q$ and $b_{q+1}$. Let
$d_i(b_1,b_2,b_3,b_4)$ be the sum of demands $d(st)$ over the pairs $st$
connecting neighboring segments $\omega_q,\omega_{q+1}$ plus twice the sum of
demands $d(st)$ over $st$ connecting either $\omega_1$ and $\omega_3$, or
$\omega_2$ and $\omega_4$.

Now for $i=1,2,3$, take as $b_1,b_2,b_3,b_4$ the edges $g(i-1),h(i-1),
h'(i+1),g'(i+1)$, respectively. One can see that the contribution to $d(m)$
from the demand pairs on $b(H_i)$ is just $d_i(g(i-1),h(i-1),
h'(i+1),g'(i+1))$. Hence
  $$ 
  d(m)=\sum\Bigl(d_i(g(i-1),h(i-1), h'(i+1),g'(i+1))\colon i=1,2,3\Bigr).
  $$

This prompts the idea to minimize $c(m)$ over a class of (2,3)-metrics $m$
which for each $i=1,2,3$, use the same quadruple of edges in $b(H_i)$, and
therefore have equal values $d(m)$. (In reality, we will be forced to include
in this class certain non-regular (2,3)-metrics as well.)

On this way we come to the following task, which is solved by comparing $O(1)$
combinations of the lengths of $c$-shortest paths in $\hat G^\ast$:
 \begin{numitem1} \label{eq:task}
Given, for $i=1,2,3$, a quadruple $\tilde Z_i=(z_i^1,z_i^2,z_i^3, z_i^4=z_i^0)$
of terminals in $Z_i$ (with possible equalities), find a set $\Pscr$ of six
(simple) paths in $\hat G^\ast$ minimizing the total $c$-length, provided that:
  \begin{itemize}
\item[($\ast$)] each path in $\Pscr$ connects terminals $z_i^p$ and $z_j^q$ with $i\ne j$,
and the set of endvertices of the paths in $\Pscr$ is exactly $\tilde Z_1\cup
\tilde Z_2\cup \tilde Z_3$ (respecting the possible multiplicities).
 \end{itemize}
  \end{numitem1}

Next we need some terminology and notation. Let $A_i$ denote the quadruple of
edges in the cycle $b(H_i)$ of $G$ that corresponds to $\tilde Z_i$ (respecting
the possible multiplicities), $i=1,2,3$. Let $\Ascr:=(A_1,A_2,A_3)$. Define
$\zeta (\Ascr)$ to be the minimum $c$-length of a path system in~\refeq{task}
and define $d(\Ascr)$ to be the corresponding combinations of demands. Then
$d(\Ascr)=d(m)$ for any $m\in\Mscr(\Ascr)$ and
  \begin{equation} \label{eq:zetaA}
  \zeta(\Ascr)\le \min\{c(m)\colon m\in \Mscr(\Ascr)\},
  \end{equation}
where $\Mscr(\Ascr)$ denote the set of regular (2,3)-metrics $m=m^\sigma$ in
$G$ \emph{agreeable to} $\Ascr$, i.e., such that for the partition
$\Xi^\sigma=(T_1,T_2,S_1,S_2,S_3)$ and $i=1,2,3$, the cuts
$\delta(S_{i-1}),\delta(S_{i+1})$ meet $b(H_i)$ by $A_i$.

In general, inequality~\refeq{zetaA} may be strong. Nevertheless, we can get a
converse inequality by extending $\Mscr(\Ascr)$ to a larger class of
(2,3)-metrics.
\medskip

\noindent\textbf{Definition.} Let us say that a (2,3)-metric $m=m^\sigma$ is
\emph{semi-regular} if the sets $S_1,S_2,S_3$ in $\Xi^\sigma$ are nonempty and
satisfy~(iii) in~\refeq{reg23}.\medskip

\noindent(Whereas $T_1,T_2$ may be empty and (ii) of~\refeq{reg23} need not
hold; in particular, subgraphs $[S_i]$ need not be connected.) We show the
following
  \begin{prop} \label{pr:semiregular}
$\zeta(\Ascr)$ is equal to $c(m)$ for some semi-regular (2,3)-metric $m$
agreeable to $\Ascr$.
  \end{prop}

(When a (2,3)-metric $m$ is semi-regular but not regular, it is ``dominated by
two cuts'', in the sense that there are $X,Y\subset V$ such that
$\Delta_{c,d}(m)\ge \Delta_{c,d}(X)+\Delta_{c,d}(Y)$, cf.~\cite[Sec.~3]{K3hI}.)
\medskip

 \begin{proof}
We use the observation that problem $\Dscr(c,d)$ remains equivalent when an
edge $e$ is subdivided into several edges in series, say, $e_1,\ldots,e_k$
($k\ge 1$) with the same capacity: $c(e_i)=c(e)$. In particular, we can
subdivide edges in the boundaries of holes, due to which we may assume that
each quadruple $A_i$ consists of different edges. Then all terminals in each
$\tilde Z_i$ become different.

Another advantage is that when considering an optimal path system $\Pscr$
in~\refeq{task}, we may assume that the paths in $\Pscr$ are pairwise
edge-disjoint. Indeed, if some edge $e^\ast$ of $\hat G^\ast$ is used by $k>1$
paths in $\Pscr$, we can subdivide the corresponding edge $e$ of $G$ into $k$
edges in series. This leads to replacing $e^\ast$ by a tuple of $k$ parallel
edges (of the same length $c(e)$) and we assign each edge to be passed by
exactly one of those paths.

We need to improve $\Pscr$ so as to get rid of ``crossings''. More precisely,
consider two paths $P,P'\in\Pscr$, suppose that they meet at a vertex $v$, let
$e,e'$ be the edges of $P$ incident to $v$, and let $g,g'$ be similar edges of
$P'$. We say that $P$ and $P'$ \emph{cross} (each other) at $v$ if $e,g,e',g'$
occur in this order (clockwise or counterclockwise) around $v$, and
\emph{touch} otherwise.

For an inner (nonterminal) vertex $v$, let $\Pscr(v)$ be the set of paths in
$\Pscr$ passing $v$, and $\Escr(v)$ the clockwise ordered set of edges incident
to $v$ and occurring in $\Pscr(v)$. We assign to the edges in $\Escr(v)$ labels
1, 2 or 3, where an edge $e$ is labeled $i$ if for the path $P\in\Pscr(v)$
containing $e$, $P$ begins or ends at a terminal $z$ in $\tilde Z_i$ and $e$
belongs to the part of $P$ between $v$ and $z$. (So if $P$ connects $\tilde
Z_i$ and $\tilde Z_j$ and $e'$ is the other edge of $P$ incident to $v$, then
$e'$ has label $j$.)

We iteratively apply the following \emph{uncrossing operation}. Choose a vertex
$v$ with $|\Escr(v)|\ge 4$. Split each path of $\Pscr(v)$ at $v$. This gives,
for each edge $e\in \Escr(v)$ with label $i$, a path containing $e$ and
connecting $v$ with a terminal in $\tilde Z_i$; denote this path by $Q(e)$.
These paths are regarded up to reversing. Now we recombine these paths into
pairs as follows, using the obvious fact that for each $i=1,2,3$, the number of
edges in $\Escr(v)$ with label $i$ is at most $|\Escr(v)|/2$.

Choose two consecutive edges $e,e'$ in $\Escr(v)$ by the following rule: $e,e'$
have different labels, say, $i,j$, and the number of edges in $\Escr(v)$ having
the third label $k$ (where $\{i,j,k\}=\{1,2,3\}$) is strictly less than
$|\Escr(v)|/2$. (Clearly such $e,e'$ exist.) We concatenate $Q(e)$ and $Q(e)$,
obtaining a path connecting $\tilde Z_i$ and $\tilde Z_j$, update
$\Escr(v):=\Escr(v)-\{e,e'\}$, apply a similar procedure to the updated
$\Escr(v)$, and so on until $\Escr(v)$ becomes empty.

One can see that the resulting path system $\Pscr'$ satisfies property~($\ast$)
in~\refeq{task} and has the same total $c$-length as before (thus yielding an
optimal solution to~\refeq{task}), and now no two paths in $\Pscr'$ cross at
$v$. Note that for some vertices $w\ne v$, edge labels in $\Escr(w)$ may become
false; this may happen with those vertices $w$ that belong to paths in
$\Pscr'(v)$. For this reason, we finish the procedure of handling $v$ by
checking such vertices $w$ and correcting their labels where needed. In
addition, if we reveal that one or another path in $\Pscr'(v)$ is not simple,
we remove the corresponding closed subpath in it (which has zero $c$-length
since $\Pscr'$ is optimal).

At the next iteration we apply a similar uncrossing operation to another vertex
$v'$, and so on. Upon termination of the process (taking $<n$ iterations) we
obtain a path system $\tilde \Pscr$ such that
 \begin{numitem1} \label{eq:nocross}
$\tilde \Pscr$ is optimal to~\refeq{task} and admits no crossings.
  \end{numitem1}

Property~($\ast$) in~\refeq{task} implies that for each $p=1,2,3$, the sets
$\tilde Z_{p-1}$ and $\tilde Z_{p+1}$ are connected by exactly two paths in
$\tilde \Pscr$. We denote them by $P_p,Q_p$ and assume that both paths go from
$\tilde Z_{p-1}$ to $\tilde Z_{p+1}$ (reversing paths in $\tilde \Pscr$ if
needed). Since $P_p,Q_p$ nowhere cross, we can subdivide the space $\Rset^2-
(\hat H_{p-1}\cup \hat H_{p+1})$ into two closed regions $\Rscr,\Rscr'$ such
that $\Rscr\cap\Rscr'=P_p\cup Q_p$, ~$\Rscr$ lies ``on the right from $P_p$''
and ``on the left from $Q_p$'', while $\Rscr'$ behaves conversely. (Here we
give informal, but intuitively clear, definitions of $\Rscr,\Rscr'$, omitting a
precise topological description.) One of them does not contain the hole $\hat
H_p$; denote it by $\Rscr_p$. We observe the following:
  \begin{numitem1} \label{eq:Rscr_p}
no path in $\tilde \Pscr$ meets the interior $\inter(\Rscr_p)$ of $\Rscr_p$.
  \end{numitem1}

Indeed, if $P\in\tilde\Pscr$ goes across $\inter(\Rscr_p)$, then $P$ is
different from $P_p$ and $Q_p$; hence $P$ has one endvertex in $\tilde Z_p$.
Since $\tilde Z_p\cap\Rscr_p=\emptyset$, ~$P$ must cross the boundary of
$\Rscr_p$. This implies that $P$ crosses some of $P_p,Q_p$, contrary
to~\refeq{nocross}.
\smallskip

From~\refeq{Rscr_p} it follows that the interiors of $\Rscr_1,\Rscr_2,\Rscr_3$
are pairwise disjoint and that for $p=1,2,3$, the paths $P_p,Q_p$ begin at
consecutive terminals in $\tilde Z_{p-1}$ and end at consecutive terminals in
$\tilde Z_{p+1}$ (assuming as before that both paths go from $\tilde Z_{p-1}$
to $\tilde Z_{p+1}$). So we may assume for definiteness that
  \begin{numitem1} \label{eq:order}
for $i=1,2,3$, the terminals $z_i^1,z_i^2,z_i^3,z_i^4$ of $\tilde Z_i$ are,
respectively, the end of $P_{i-1}$, the end of $Q_{i-1}$, the beginning of
$Q_{i+1}$, and the beginning of $P_{i+1}$;
  \end{numitem1}
see the picture, where for simplicity all paths are vertex disjoint.

\vspace{-0.3cm}
\begin{center}
\includegraphics{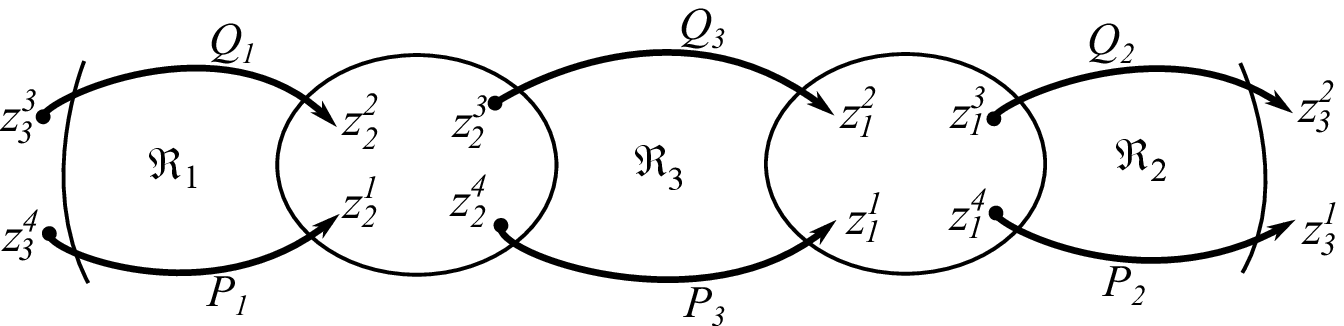}
\end{center}
\vspace{-0.3cm}

Then the space $\Rset^2- (\hat H_1\cup \hat H_2\cup \hat H_3\cup
\inter(\Rscr_1)\cup \inter(\Rscr_2)\cup\inter(\Rscr_3))$ can be subdivided into
two closed regions $\Lscr_1$ and $\Lscr_2$, where the former lies ``on the
right from $P_1,P_2,P_3$'' and the latter lies ``on the left from
$Q_1,Q_2,Q_3$''. One can see that
  \begin{numitem1} \label{eq:RL}
each edge of $P_p$ is shared by the regions $\Rscr_p$ and $\Lscr_1$, and each
edge of $Q_p$ is shared by $\Rscr_p$ and $\Lscr_2$.
  \end{numitem1}

Now the sets of faces in (the natural extensions to $G^\ast$ of) the regions
$\Lscr_1,\Lscr_2,\Rscr_1,\Rscr_2,\Rscr_3$ induce vertex sets
$T_1,T_2,S_1,S_2,S_3$ in $G$, respectively, giving a partition of $V$. Let $m$
be the (2,3)-metric determined by this partition. Then~\refeq{order} implies
that $m$ is semi-regular and agreeable to $\Ascr$. By~\refeq{RL}, for
$p=1,2,3$, each edge of $\delta(S_p)$ connects $S_p$ with one of $T_1,T_2$
(whereas no edge of $G$ connects $T_1$ and $T_2$, or connects $S_i$ and $S_j$
for $i\ne j$). Therefore,
  $$
  \zeta(\Ascr)=\sum(c(P_p)+c(Q_p)\colon p=1,2,3)=c(m),
  $$
yielding the proposition.
  \end{proof}

\noindent\textbf{Remark 4.} Strictly speaking, the metric $m$ in the above
proof concerns the modified graph, obtained by replacing some edges $e=uv$ of
the original graph $G$ by paths $L_e$ connecting $u$ and $v$. When returning to
the original $G$, those elements of $S_p$ or $T_q$ that are intermediate
vertices of such paths $L_e$ disappear, and as a result, there may appear
(original) edge connecting $T_1$ and $T_2$, or $S_i$ and $S_j$, $i\ne j$. One
can see, however, that this does not affect the value $c(m)$ for the
corresponding $m$.
\medskip

Finally, define $\tilde\Delta_{c,d}(\Ascr):=\zeta(\Ascr)-d(\Ascr)$. We conclude
with the following
  \begin{corollary} \label{cor:Delta23}
{\rm(i)} Let $\Ascr=(A_1,A_2,A_3)$, where $A_i$ is a quadruple of edges in
$b(H_i)$. Then $\Delta_{c,d}(m)\ge\tilde\Delta_{c,d}(\Ascr)$ for each regular
(2,3)-metric $m$ agreeable to $\Ascr$, and there exists a semi-regular
(2,3)-metric $m'$ agreeable to $\Ascr$ such that
$\Delta_{c,d}(m')=\tilde\Delta_{c,d}(\Ascr)$. In particular, if
$\tilde\Delta_{c,d}(\Ascr)<0$, then problem $\Dscr(c,d)$ has no solution.

{\rm(ii)} The minimum $\hat\mu_{c,d}$ of excesses $\Delta_{c,d}(m)$ over the
semi-regular (2,3)-metrics $m$ can be found in $O(n^{12}+n\cdot SP(n))$ time.
  \end{corollary}


\section{Algorithm} \label{sec:alg}

As before, we assume that the capacity-demand pair $(c,d)$ is Eulerian.

The algorithm starts with verifying cut condition~\refeq{cut_cond} and
(2,3)-metric condition for the initial problem $\Dscr(G,\Hscr,D,c,d)$, using
the efficient procedures described in Sections~\SEC{veri_cut}
and~\SEC{veri_23}. If some condition is violated, we declare that the problem
has no solution. Otherwise the algorithm recursively constructs an integer
admissible multiflow. We may assume, w.l.o.g., that all current capacities and
demands are nonzero (for edges $e$ with $c(e)=0$ can be immediately deleted
from $G$, and similarly for pairs $st\in D$ with $d(st)=0$), and that the
boundary $b(H_i)$ of each hole $H_i$ is connected and isthmusless, regarding it
as a cycle.

An \emph{iteration} of the algorithm applied to current $G,\Hscr,D,c,d$ (with
$(c,d)$ Eulerian) chooses arbitrary $i\in\{1,2,3\}$, an edge $e=uv$ in
$b(H_i)$, and a pair $st\in D_i$ (where $D_i$ denotes the set of demand pairs
for $H_i$).

Let for definiteness $s,u,v,t$ follow in this order in $b(H_i)$. For an integer
$\eps\le\min\{c(e),d(st)\}$, let us transform $(c,d)$ into the capacity-demand
pair $(c',d')$ by
  \begin{gather} \label{eq:split}
  c'(e):= c(e)-\eps, \quad d'(st):=d(st)-\eps, \\
  d'(su):=d(su)+\eps, \quad\mbox{and} \quad d'(vt):=d(vt)+\eps. \nonumber
   \end{gather}
(Here we add to $D$ the demand pair $su$ with $d(su):= 0$ if it does not exist
there, and similarly for $vt$. When $s=u$ ($v=t$), the pair $su$ (resp. $vt$)
vanishes.) Clearly $(c',d')$ is Eulerian as well. We say that $(c',d')$ is
obtained by the $(e,st,\eps)$-\emph{reduction} of $(c,d)$. We call $\eps$ a
\emph{feasible reduction number} for $c,d,e,st$, or, simply, \emph{feasible},
if the problem $\Dscr(c',d')$ is still solvable (and therefore it has an
integer solution). The goal of the iteration is to find the maximum feasible
$\eps$ and then update $c,d$ accordingly.

Here we rely on the existence of an evident transformation of an integer
admissible multiflow $f'$ for $(c',d')$ into an integer admissible multiflow
$f$ for $(c,d)$: extract from $f'$ an integer subflow $g$ from $s$ to $u$ and
an integer subflow $h$ from $v$ to $t$, of value $\eps$ each, and increase the
flow between $s$ and $t$ by concatenating $g,h$ and the flow of value $\eps$
through the edge $e$.

The procedure of finding the maximum feasible $\eps$ consists of $O(1)$ steps.
We use a consequence from assertions in Sections~\SEC{veri_cut},~\SEC{veri_23}
(using notation from these sections).
  \begin{prop} \label{pr:nu}
Let $\eps\le c(e),d(st)$ and let $(c',d')$ be obtained by the
$(e,st,\eps)$-reduction of $(c,d)$. Suppose that $\Dscr(c,d)$ is solvable but
$\Dscr(c',d')$ is not. Let $\tilde\nu$ be the minimum of $\mu^1_{c',d'},
\nu^2_{c',d'}, \nu^3_{c',d'}, \hat\mu_{c',d'}$. Then $\tilde\nu$ is the minimum
excess for $(c',d')$ among the regular sets and semi-regular (2,3)-metrics.
  \end{prop}
  \begin{proof}
Since $\Dscr(c',d')$ has no solution, at least one of $\mu^\alpha_{c',d'}$,
$\alpha=1,2,3$, and $\hat\mu_{c',d'}$ is negative (by
Theorems~\ref{tm:3hole},~\ref{tm:3h_sharp}). Also $\nu^\alpha_{c',d'}\le
\mu^\alpha_{c',d'}$ for $\alpha=2,3$ (by
Propositions~\ref{pr:type2},~\ref{pr:type3}). Hence $\tilde \nu<0$.

Suppose that $\tilde\nu=\nu^2_{c',d'}<\mu^2_{c',d'}$. Then we are in case~(ii)
of Proposition~\ref{pr:type2}; let $X,Y$ be as in this case. Then $X$ and $Y$
are of type 1 and concern different holes.  Therefore, under the transformation
$(c,d)\mapsto(c',d')$ the excess of one of $X,Y$ does not change. Indeed,
assuming for definiteness that $\Hscr(X)\ne \{H_i\}$ (where, as before, $e,s,t$
are in $b(H_i)$), we observe that none of $uv,st,su,vt$ is separated by $X$,
whence $\Delta_{c',d'}(X)=\Delta_{c,d}(X)$. Now since $\Delta_{c',d'}(X)+
\Delta_{c',d'}(Y)\le \nu^2_{c',d'}$ and $\Delta_{c,d}(X)\ge 0$ (as $\Dscr(c,d)$
is solvable), we have $\Delta_{c',d'}(Y)\le\nu^2_{c',d'}$. But then
$\mu^1_{c',d'}\le \Delta_{c',d'}(Y)$ and $\nu^2_{c',d'}=\tilde\nu\le
\mu^1_{c',d'}$ imply $\tilde \nu=\mu^1_{c',d'}$.

Next suppose that $\tilde \nu=\nu^3_{c',d'}<\mu^3_{c',d'}$. Then we are in
case~(ii) or~(iii) of Proposition~\ref{pr:type3}. Arguing as above, we can
conclude that there is a regular set $X'$ (which is one of $X_1,X_2,X_3$ in
case~(ii), and one of $X,Y$ in case~(iii)) such that
$\Delta_{c',d'}(X')\le\nu^3_{c',d'}$. This implies that
$\tilde\nu=\mu^\alpha_{c',d'}$ for some $\alpha\in\{1,2\}$.

Thus, in all cases we obtain $\tilde\nu=\min\{\mu^1_{c',d'}, \mu^2_{c',d'},
\mu^3_{c',d'}, \hat\mu_{c',d'}\}$.
  \end{proof}

The maximum feasible $\eps$ is computed in at most three steps. First we try to
take as $\eps$ the maximum possible value, namely,
$\eps_1:=\min\{c(e),d(st)\}$; let $c_1,d_1$ be defined as in~\refeq{split} for
this $\eps_1$. We determine the number $\tilde \nu_1:=\min\{\mu^1_{c_1,d_1},
\mu^2_{c_1,d_1}, \mu^3_{c_1,d_1}, \hat\mu_{c_1,d_1}\}$ (using procedures from
Sections~\SEC{veri_cut},\SEC{veri_23} and relying on Proposition~\ref{pr:nu}).
If $\tilde\nu_1\ge 0$ then $\eps:=\eps_1$ is as required.

And if $\tilde \nu_1<0$, we take  $\eps_2:=\eps_1+\lfloor \tilde \nu_1
/4\rfloor$, define $c_2,d_2$ as in~\refeq{split} for this $\eps_2$ and $c,d$ as
before, and find $\tilde \nu_2:=\min\{\mu^1_{c_2,d_2}, \mu^2_{c_2,d_2},
\mu^3_{c_2,d_2}, \hat\mu_{c_2,d_2}\}$ (step 2). Again, if $\tilde \nu_2\ge 0$
then $\eps_2$ is just the desired $\eps$.

Finally, if $\tilde \nu_2< 0$, we take as $\eps$ the number
$\eps_3:=\eps_2+\tilde \nu_2/2$ (step 3).
  \begin{lemma} \label{lm:3steps}
The $\eps$ determined in this way is indeed the maximum feasible reduction
number for $c,d,e,st$.
  \end{lemma}
  \begin{proof}
We argue in a similar way as for an integer splitting in~\cite{K87}. For a
regular set $X\subset V$, define $\beta(X):=\omega_X(s,u)+\omega_X(u,v)
+\omega_X(v,t)-\omega_X(s,t)$, where we set $\omega_X(x,y):=1$ if $X$ separates
vertices $x$ and $y$, and 0 otherwise. For a semi-regular (2,3)-metric $m$,
define $\gamma(m):=m(su)+m(uv)+m(vt)-m(st)$. Then $\beta(X)\ge 0$ and
$\gamma(m)\ge 0$ (since both $\omega_X$ and $m$ are metrics). Moreover, one can
check that $\beta(X)\in\{0,2\}$ and ~$\gamma(m)\in\{0,2,4\}$, and that if
$(c'',d'')$ is obtained by the $(e,st,\eps')$-reduction of $(c',d')$ for an
arbitrary $\eps'$, then
  \begin{equation} \label{eq:new_excess}
  \Delta_{c'',d''}(X)=\Delta_{c',d'}(X)-\eps'\beta(X) \quad\mbox{and}\quad
  \Delta_{c'',d''}(m)=\Delta_{c',d'}(m)-\eps'\gamma(m).
  \end{equation}

Let $\bar \eps$ be the maximum feasible reduction number for $c,d,e,st$. When
$\tilde\nu_1\ge 0$, the equality $\bar \eps=\eps_1$ is obvious, so suppose that
$\tilde\nu_1<0$. If $\tilde\nu_1$ is achieved by the excess (w.r.t. $c_1,d_1$)
of a semi-regular (2,3)-metric $m$ and if $\gamma(m)=4$, then using the second
expression in~\refeq{new_excess} and the equality $\eps_2=\eps_1+\lfloor \tilde
\nu_1 /4\rfloor$, we have
  \begin{multline*}
  \Delta_{c_2,d_2}(m)=\Delta_{c,d}(m)-\eps_2\gamma(m)
   =\Delta_{c,d}(m)-\eps_1\gamma(m)-\lfloor\tilde\nu_1/4\rfloor\cdot 4 \\
   =\Delta_{c_1,d_1}(m)-\lfloor\tilde\nu_1/4\rfloor\cdot 4
      =\tilde\nu_1-\lfloor\tilde\nu_1/4\rfloor\cdot 4=\tau,
  \end{multline*}
where $\tau$ equals 0 if $\tilde\nu_1$ is divided by 4, and equals 2 otherwise.
(Recall that the excess of any (2,3)-metric is even when the capacity-demand
pair is Eulerian.) In this case we have $\bar\eps\le\eps_2$. Indeed for
$\eps':=\eps_2+1$, the pair $(c',d')$ obtained by the $(e,st,\eps')$-reduction
of $(c,d)$ would give $\Delta_{c',d'}(m)=\Delta_{c_2,d_2}(m)-4<0$; so $\eps'$
is infeasible.

As a consequence, in case $\tilde\nu_2\ge 0$ we obtain $\bar\eps=\eps_2$.

Now let $\tilde\nu_2< 0$. Note that for any semi-regular metric $m'$ with
$\gamma(m')=4$, the facts that $\gamma(m')=\gamma(m)$ and
$\Delta_{c_1,d_1}(m')\ge\tilde\nu_1=\Delta_{c_1,d_1}(m)$ imply that
$\Delta_{c',d'}(m')\ge\Delta_{c',d'}(m)\ge 0$ for any $(c',d')$ obtained by the
$(e,st,\eps')$-reduction of $(c,d)$ with $\eps'\le \eps_2$. Therefore,
$\tilde\nu_2$ is achieved by either a set $X$ with $\beta(X)=2$ or a
semi-regular (2,3)-metric $m''$ with $\gamma(m'')=2$. This implies
$\bar\eps=\eps_2+\tilde \nu_2/2$.
  \end{proof}

Also the above procedure of computing $\eps$ together with the complexity
results in Sections~\SEC{veri_cut} and~\SEC{veri_23} gives the following
  \begin{corollary} \label{cor:iter}
Each iteration (finding the corresponding maximum reduction number and reducing
$c,d$ accordingly) takes $O(n^{12})$ time.
  \end{corollary}

Next, considering~\refeq{new_excess} and using the facts that
$\beta(X),\gamma(m)\ge 0$, we can conclude that under a reduction as above the
excess of any set or (2,3)-metric does not increase. This implies that
  \begin{numitem1} \label{eq:crit}
if an iteration handles $c,d,e,st$, then for any capacity-demands $(c',d')$
arising on subsequent iterations, the maximum reduction number for $c',d',e,st$
is zero.
  \end{numitem1}

Therefore, it suffices to consider each pair $(e,st)$ at most once during the
process.

Now we finish our description as follows. Suppose that, at an iteration with
$i,e,st$, the capacity of $e$ becomes zero and the deletion of $e$ from $G$
causes merging $H_i$ with another hole $H_j$. Then we can proceed with an
efficient procedure for solving the corresponding Eulerian 2-hole demand
problem. Similarly, if the demand on $st$ becomes zero and if the deletion of
$st$ makes $D_i$ empty, then we can withdraw the hole $H_i$, again obtaining
the Eulerian 2-hole case.

Finally, suppose that we have the situation when for some $c,d$, the holes
$H_1,H_2,H_3$ are different (and the capacities of all edges are positive),
each $D_1,D_2,D_3$ is nonempty, but the maximum feasible reduction number for
any corresponding pair $e,st$ is zero. We assert that this is not the case.

Indeed, suppose such $c,d$ exist. The problem $\Dscr(c,d)$ is solvable, and one
easily shows that there exists an integer solution $f=(\Pscr,\lambda)$ to it
such that: for some path $P\in\Pscr$ with $\lambda(P)>0$, some edge $e$ of $P$
belongs to the boundary of the same hole $H_i$ that contains the ends
$s_P,t_P$. But this implies that $s_Pt_P\in D_i$ and that $\eps=1$ is feasible
for $c,d,e,s_Pt_P$; a contradiction.

Thus, we obtain the following
  \begin{theorem} \label{tm:alg}
The above algorithm terminates in $O(n^3)$ iterations and finds an integer
solution to $\Dscr(G,\Hscr,D,c,d)$ with $|\Hscr|=3$ and $(c,d)$ Eulerian.
  \end{theorem}

In conclusion of this paper, recall that when $|\Hscr|=4$ and $(c,d)$ is
Eulerian, the solvability of $\Dscr(c,d)$ implies the existence of a
\emph{half-integer} solution (see (C) in the Introduction). An \emph{open
question}: does there exist a polynomial-time (not necessarily ``purely
combinatorial'') algorithm to find such a solution? (Note that the solvability
of $\Dscr(c,d)$ can be verified in strongly polynomial time, by using a version
of ellipsoid method.)

\end{document}